\documentclass[11pt]{amsart}
\usepackage[utf8]{inputenc}

\usepackage[plainpages=false]{hyperref}
\usepackage{amsfonts,latexsym,rawfonts,amsmath,amssymb,amsthm, mathrsfs, lscape}
\usepackage{verbatim}

\usepackage[all]{xy}
\usepackage{graphicx,psfrag}

\usepackage{array,tabularx}

\usepackage{setspace}
\usepackage{anysize}%*

\newtheorem{thm}{Theorem}
\newtheorem{cor}{Corollary}[section]

\newtheorem{prop}{Proposition}[section]

\newtheorem{conj}[thm]{Conjecture}

\theoremstyle{remark}
\newtheorem{rmk}{Remark}[section]

\theoremstyle{definition}
\newtheorem{defn}{Definition}[section]

\numberwithin{equation}{section}

\newtheorem*{theorem*}{Theorem}

\newtheorem{theorem}{Theorem}[section]

\newtheorem{lemma}[theorem]{Lemma}

\marginsize{2.6cm}{2.6cm}{1cm}{2cm}
\def\p{\partial}
\def\R{\mathbb{R}}
\def\C{\mathbb{C}}
\def\N{\mathbb{N}}

\def\l{\lambda}

\def\cD{\mathcal D}
\def\cE{{\mathcal E}}

\def\cH{{\mathcal H}}

\def\cK{{\mathcal K}}
\def\cL{{\mathcal L}}

\def\cO{{\mathcal O}}

\begin{document}

\title{On Calabi's Extremal metrics and properness}
\author{Weiyong He}
\address{Department of Mathematics, University of Oregon, Eugene, OR 97403. }
\email{whe@uoregon.edu}

\begin{abstract}
In this paper we  extend recent breakthrough of Chen-Cheng \cite{CC1, CC2, CC3} on existence of constant scalar K\"ahler metric on a compact K\"ahler manifold to Calabi's extremal metric. Our argument follows \cite{CC3} and there are no new a prior estimates needed, but rather there are necessary modifications adapted to the extremal case. 
We prove that there exists an extremal metric with extremal vector $V$  if and only if the modified Mabuchi energy is proper, modulo the action the subgroup in the identity component of automorphism group which commutes with the flow of $V$. 
We introduce two essentially equivalent notions, called \emph{reductive properness} and \emph{reduced properness}. We observe that one can test reductive properness/reduced properness only for invariant metrics. We prove that existence of an extremal metric is equivalent to reductive properness/reduced properness of the modified Mabuchi energy. 
\end{abstract}

\maketitle

\section{Introduction}
Recently Chen and Cheng \cite{CC1, CC2, CC3} have made surprising breakthrough in the study of constant scalar curvature (csck) on a compact K\"ahler manifold $(M, [\omega], J)$, with a fixed K\"ahler class $[\omega]$. A K\"ahler metric with constant scalar curvature are special cases of Calabi's extremal metric \cite{Ca1, Ca2}, which was introduced by E. Calabi in 1980s to find  a canonical representative within the K\"ahler class $[\omega]$.
In \cite{Don97}, S. K. Donaldson proposed a beautiful program in K\"ahler geometry to attack Calabi's renowned problem of existence of csck metrics in terms of the geometry of space of K\"ahler potentials $\cH$ with the Mabuchi metric \cite{M1}, 
\[\cH:=\{\phi\in C^\infty(M): \omega_\phi=\omega+\sqrt{-1} \p\bar\p \phi>0\}.\]
Chen's resolution \cite{chen01} of Donaldson's conjecture that $\cH$ is a metric space, where any two points in $\cH$ can be joint be a unique $C^{1, 1}$ geodesic would play an important role in Donaldson's program.
An important notion is Mabuchi's $\cK$-energy \cite{M1, M2}, which is convex along $C^{1, 1}$ geodesics in $\cH$ \cite{BB} and csck is a critical point of $\cK$-energy.  Mimicking finite dimensional setting,  criterions of existence of a critical point of a convex functional have been proposed in this infinitely dimensional setting, which include Donaldson's infinite dimensional geometric invariant theory (GIT), which leads to his conjecture \cite{Don97} by examining the derivatives of $\cK$-energy at infinity of geodesic rays and Chen-Donaldson's geodesic stability conjecture,  and Chen's conjecture  \cite{chen01} of properness of $\cK$-energy in terms of distance function of $\cH$. We should also mention the properness conjecture made by G. Tian \cite{Tian01}, where  the properness of $\cK$-energy was proposed in terms of Aubin's $J$-functional and Tian's conjecture was partly motivated by conformal geometry. For a more detailed historic review for existence of csck and extremal metric, including the discussion of Yau-Tian-Donaldson stability conjecture for projective manifolds, we refer to \cite{PS} and \cite{CC2}[Introduction]. 

Recently Berman-Darvas-Lu \cite{BDL, BDL2} proved that the existence of csck implies the properness of $\cK$-energy (and K-polystability). This work relies on recent developments in K\"ahler geometry, including Darvas' deep understanding \cite{D1, D2} of the geometric completion of $\cH$ with respect to not only Mabuchi's metric, but also the Finsler metric $d_1$,  and the convexity of $\cK$-energy \cite{BB} (see also \cite{CLP}),  and results in \cite{DR}. 

For the existence part, in the case of toric surface, S. Donaldson \cite{Don02, Don07, Don09} proved that existence of csck is equivalent to K-stability (as well as properness). Donaldson's approach relied on toric surface assumption where the dimension and convex analysis play essential roles.
Other than Donaldson's deep work, how to approach existence of csck seemed intractable technically in general.
One of the key steps is to derive the estimate of second order (from point view of PDE), or to control the metric in terms of its scalar curvature (from point of view of geometry). This problem has been known in the field to be extremely hard. 
The csck equation reads, given a K\"ahler potential such that $g_{i\bar j}+\phi_{i\bar j}>0$, 
\[
-g^{i\bar j}_\phi\p_i\p_{\bar j}\log \det(g_{i\bar j}+\phi_{i\bar j})=\text{constant}
\]
This is a formidable fourth order nonlinear elliptic equation on $\phi$. How to derive its a priori estimates, in particular to control $n+\Delta \phi$,  seemed to be totally out of sight.  An attempt made by Chen and the author \cite{chenhe12} is to derive second order estimate through the complex Monge-Ampere equation under certain ``minimal" assumption on ratio of the volume forms, with the motivation to derive second order estimates in terms of csck equation. 
The classical results  of Yau \cite{Yau} and Aubin \cite{Aubin} using maximum principle to control $n+\Delta\phi$ through the complex Monge-Ampere equation require  regularity of the volume ratio up to $C^2$.  As a comparison, together with Chen in \cite{chenhe12} 
we have obtained a rather optimal regularity depending on $W^{1, p}$-norm of the volume ratio, where $p>2\text{dim}_\R M$ is the optimal power. 
We also mention the work together with Chen and Darvas \cite{CDH}, where we obtained the control of second order ($n+\Delta\phi$), assuming merely $d_1$-distance bound and upper bound for Ricci curvature. This result builds upon deep analysis in pluripotential theory and the complex Monge-Ampere equation. Using this result, we have proved the properness of $\cK$-energy implies the existence of csck along Chen's continuity path \cite{chen15} if we assume Ricci curvature upper bound \cite{CDH}[Theorem 1.4]. This seemed to summarize the author's knowledge of limitation to control the second order through its scalar curvature (see in particular the discussion in \cite{CDH} after Theorem 1.4 and Remark 2.1).\\

Here come the breakthrough of Chen-Cheng \cite{CC1, CC2, CC3}. In short, Chen and Cheng have obtained the a priori estimates by treating csck type equations as two coupled second order elliptic equations. Their methods are very skillful and truly amazing. Using these estimates and recent developments in K\"ahler geometry, they proved the following theorem,
\begin{thm}[Chen-Cheng]\label{main0}Suppose the $\cK$-energy is proper with respect to $d_1$ modulo $\text{Aut}_0(M)$, then there exists a smooth csck in $(M, [\omega], J)$. 
\end{thm}
Chen and Cheng  proved much more, including Donaldson's conjecture \cite{Don97} and Chen-Donaldson geodesic stability conjecture. We refer the readers to \cite{CC1, CC2, CC3} for their exciting achievements. 
In this paper we deal with Calabi's extremal metrics, by extending Chen-Cheng's breakthrough \cite{CC1, CC2, CC3} on the existence of constant scalar curvature metric (csck) in $(M, [\omega], J)$ to the extremal case.  We prove the following, 
\begin{thm}\label{main1} Let $V$ be an extremal vector field of $(M, [\omega], J)$, which is unique up to the choice of a maximal compact subgroup of $\text{Aut}_0(M)$. 
Then the following statements are equivalent,
\begin{enumerate}
\item There exists an extremal metric in $(M, [\omega], J)$.
\item The modified $\cK$-energy is proper with respect to $d_{1}$ modulo the action of $\text{Aut}_0(M, V)$, the subgroup of $\text{Aut}_0(M)$ which commutes with the flow of $V$. \end{enumerate}
\end{thm}

\begin{rmk}By results of \cite{D2} and \cite{DR}, the properness  in terms of $d_1$ and the properness  in terms of Aubin's $J$-functional are equivalent. When $V=0$, Theorem \ref{main1} is reduced to Theorem \ref{main0} and \cite{BDL2}[Theorem 1.3]. \end{rmk}

How to formulate a precise notion of properness for $\cK$-energy is indeed a subtle issue when $\text{Aut}_0(M)\neq 0$. This was clarified by Zhou-Zhu\cite{ZZ1} and Darvas-Rubinstein \cite{DR}. In particular, they disproved the following conjecture \cite{DR},
\begin{conj}\label{c200}Let $(M, \omega, J)$ be a compact K\"ahler manifold. Let $K$ be a maximal compact subgroup of $\text{Aut}_0(M)$. Then $\cH$ contains a constant scalar curvature metric if and only if $\cK$-energy is proper on the subset $\cH_K\subset \cH$ consisting of $K$-invariant metrics. 
\end{conj}

It is usually much more effective to check properness for metrics with sufficient symmetry. Our observation is that, one can still restrict to $K$-invariant metrics, for a maximal compact subgroup of $\text{Aut}_0(M, V)$, if we assume that $\text{Aut}_0(M, V)$ is the complexification of $K$. Note that $\text{Aut}_0(M, V)$ has to be the complexification of its maximal compact subgroup if there exists an extremal metric by Calabi \cite{Ca2} (see Theorem \ref{calabi01}). We write $G=\text{Aut}_0(M, V)$ for simplicity.  We shall also consider a compact subgroup $K_0\subset K$ such that the Lie algebra of $K_0$ consisting of reduced holomorphic vector fields and let $G_0$ be the complexification of $K_0$. 
We introduce the notion reductive-properness (for the pair $(K, G)$) and reduced properness (for the pair $(K_0, G_0)$), see Definition \ref{rp}, Definition \ref{rp1} and Definition \ref{rp2}. 
These two definitions are essentially the same, in view of Calabi's decomposition \eqref{decom01}.
Reductive properness and reduced properness seem to be  right replacements of properness needed in Conjecture \ref{c200}.  We have the following
 
\begin{thm}\label{main2}
The following statements are equivalent,
\begin{enumerate}
\item There exists an extremal metric in $(M, [\omega], J)$ with extremal vector field $V$.
\item The $\cK_X$-energy is reductive-proper restricted to $\cH_K$ with respect to $d_{1, G}$. 
\item The $\cK_X$-energy is reduced-proper restricted to $\cH_{K_0}$ with respect to $d_{1, G_0}$. 
\end{enumerate}
\end{thm}

{\bf Acknowledgement:} The author thanks Prof. Xiuxiong Chen sincerely for his kindness to ask him to write down the extremal case. The author also thanks Song Sun for insightful discussions. The author is partly supported by an NSF grant, award no. 1611797. 

\numberwithin{equation}{section}
\numberwithin{thm}{section}

\section{Preliminaries}
We recall some basic concepts in K\"ahler geometry, in particular those related to Calabi's extremal metric \cite{Ca1, Ca2}. 
A good reference for extremal metrics may be found in Gauduchon \cite{PG}[Chapter 2, 3, 4].  
We consider the space of normalized K\"ahler potentials, which can be identified with the space of K\"ahler metrics in the class class $[\omega]$,
\begin{equation}
\cH_0=\{\phi\in C^\infty(M): \omega_\phi=\omega+\sqrt{-1} \p\bar\p \phi>0, \mathbb{I}_\omega(\phi)=0\},
\end{equation}
where the functional $\mathbb{I}_{\omega}$ is defined to be
\begin{equation}
\mathbb{I}_\omega(\phi)=\frac{1}{(n+1)!}\int_M \phi \sum_{k=0}^n \omega^k_\phi\wedge \omega^{n-k}.
\end{equation}
Note that the $\mathbb{I}$-functional can be characterized as
$\delta \mathbb{I}=\int_M\delta\phi \omega_\phi^n$. We also define the $\mathbb{I}_{\omega, \chi}$ functional, for a $(1, 1)$-form $\chi$,
\begin{equation}
\mathbb{I}_{\omega, \chi}(\phi)=\frac{1}{n!}\int_M \sum_{k=0}^{n-1}\phi \chi \wedge\omega_\phi^k\wedge\omega^{n-1-k}
\end{equation}
When the dependence on $\omega$ is clear, we simply write the functional $\mathbb{I}(\phi)$ and $\mathbb{I}_\chi(\phi)$. We should mention that $\mathbb{I}$ is a functional defined on $\cH$, which depends a particular normalization, for example, $\mathbb{I}(\phi+c)=\mathbb{I}(\phi)+c\text{Vol}(M)$. \\

We need some preparations on the automorphism group of $(M, J)$. 
Let $\text{Aut}_0(M)$ be the identity component of the automorphism group and $\mathfrak{h}$ be its Lie algebra. Calabi's extremal metric is defined to be that $\nabla R$ is a real holomorphic vector field, hence $J\nabla R$ is a Killing vector field (of the extremal metric). We recall the following result of Calabi,
\begin{thm}[Calabi]\label{calabi01}Let $(M, g, J)$ be an extremal K\"ahler metric. The complex Lie algebra $\mathfrak{h}$ has the following orthogonal decomposition
\begin{equation}\label{decom01}\mathfrak{h}=\mathfrak{h}^0\oplus (\oplus_{\l>0} \mathfrak{h}^\l),
\end{equation} 
where $\mathfrak{h}^0$ denotes the kernel of $[J\nabla R, \cdot]$, i.e. the centralizer of $J\nabla R$ in $\mathfrak{h}$, and for any $\l>0$, $\mathfrak{h}^\l$ denotes the subspace of elements $X$ of $\mathfrak{h}$ such that $[J\nabla R, X]=\l JX$. In particular, the identity component of isometry group $\text{Isom}_0(g)$ of $g$ is a maximal compact subgroup of $\text{Aut}_0(M)$.  
 \end{thm}
 Calabi's theorem is a generalization of Matsushima's  results on K\"ahler-Einstein metrics \cite{Mat} and Lichnerowicz's results \cite{Li} on constant scalar curvature metrics. 
 A point to make is that an \emph{extremal vector field} is \emph{a priori} determined by $(M, [\omega], J)$ due to Futaki-Mabuchi \cite{FM}, regardless whether an extremal metric exists or not  (as a consequence of \emph{Futaki-Mabuchi bilinear form}). Such a real holomorphic vector field is unique up to the conjugate action of $\text{Aut}_0(M)$. 
Hence the Futaki invariant \cite{Fut} is zero if and only if $V=0$. By Calabi's theorem, if an extremal exists, then $\text{Aut}_0(M, V)$ corresponds to the Lie algebra $\mathfrak{h}^0$, hence it is the complexification of a maximal compact subgroup $K$. In particular, we have the decomposition
\begin{equation}\label{calabi002}
\mathfrak{h}^0=\mathfrak{a}\oplus \mathfrak{ham}\oplus J\mathfrak{ham},
\end{equation} 
where $\mathfrak{ham}$ is the Lie algebra of reduced (holomorphic) Killing vector field, and $\mathfrak{a}$ is the Abelian algebra of parallel holomorphic vector fields. The Lie algebra of $K$, $\mathfrak{k}=\mathfrak{a}\oplus \mathfrak{ham}$ as in the decomposition \eqref{calabi002}. 

Given the (real) extremal vector field $V$, we refer the corresponding holomorphic vector field $X=V-iJV$  as the (complex) extremal vector field.
We assume the background metric $\omega$ is invariant, $L_{JV}\omega=0$. It follows that $X$ has a real \emph{holomorphic potential}, denoted as $\theta_X$, such that
\[
X=g^{j\bar k} \frac{\p \theta_X}{\p \bar z_k}\frac{\p}{\p z_j},
\]
where $\theta_X$ is normalized such that $\int_M \theta_X \omega^n=0$
We are interested in the \emph{invariant metrics}.  Consider the normalized space of K\"ahler potentials which is invariant under the action of $JV$, 
\[
\cH_{X}=\{\phi\in \cH: L_{JV}\phi=0\}
\]
Note that $JV$  induces a compact subtorus $T$ in $G$ and hence $\cH_X$ denotes the potentials which are invariant under the $T$-action. We also use the notation
$\cH^0_X=\cH_X\cap \cH_0$. In general, for any compact subgroup of $\text{Aut}_0$, denote $\cH_P$ be $P$-invariant elements in $\cH_X$.
We also mention that $G$ acts on $\cH_0$ and it preserves $\cH_X^0:=\cH_X\cap \mathbb{I}^{-1}(0)$. 
For any $\phi\in \cH_X$, we denote the corresponding holomorphic potential as $\theta_X(\phi)$ and we have the following
\begin{equation}\label{hpotential}
\theta_X(\phi)=\theta_X+V(\phi)\end{equation}
It is straightforward to check that the normalization condition holds given \eqref{hpotential}, \[\int_M \theta_X(\phi) \omega_\phi^n=0.\]
A metric $\omega_\phi$ is extremal if and only if it satisfies the equation
\begin{equation}\label{extremal1}
R_\phi-\underline{R}=\theta_X(\phi). 
\end{equation}

\subsection{The metric completion $(\overline\cH, d_1)$}
Generalizing the Mabuchi metric on $\cH$ (which corresponds to $d_2$ in Darvas' setting), T. Darvas \cite{D1} introduced the notion of $d_1$ distance on $\cH$ (more generally $d_p$ for $p\in [1, \infty)$),
\[
\|\xi\|_{1, \phi}=\int_M |\xi|\frac{\omega_\phi^n}{n!}, \forall \xi\in T_\phi\cH=C^\infty(M). 
\]
One can then define the length of curves in $\cH$ and the distance function $d_1(\phi_0, \phi_1)$ for $\phi_0, \phi_1$ in $\cH$. The later is imply the infimum of the length of the curves connecting $\phi_0$ and $\phi_1$.
Darvas proved that $(\cH, d_1)$ is a metric space, built upon Chen's result \cite{chen00} that $(\cH, d_2)$ (with the Mabuchi metric) is a metric space. 
Guegj and Zeriahi \cite{GZ} introduced the finite energy space of $\omega$-plurisubharmonic functions, for any $p\geq 1$,
\begin{equation}
\cE^p=\{\phi\in\text{psh}(M, \omega): \int_M |\phi|^p\omega_\phi^n<\infty, \int_M \omega_\phi^n=\int_M \omega^n\}
\end{equation}
V. Guedj conjectured \cite{G1} that the metric completion with the Mabuchi metric is precisely $\cE^2$. T. Darvas \cite{D1, D2} proved Guedj's conjecture. Indeed, he proved that the metric completion of $\cH$ with respect to $d_p$ for any $p$  is $\cE^p$, with the following characterization for $(\cE^p, d_p)$, which we only state for $p=1$ for our purpose, 

\begin{thm}[Darvas]\label{dar1}There exists a constant $C>1$ depending only on $n$ such that
\begin{equation}
C^{-1}I(u, v)\leq d_1(u, v)\leq C I_1(u, v), u, v\in \cH
\end{equation}
where we define
\[
I_1(u, v)=\int_M |u-v|\left(\frac{\omega_u^n}{n!}+\frac{\omega_v^n}{n!}\right)
\]
\end{thm}
We will also use the following notations
\[
\cE^1_0=\cE^1\cap \mathbb{I}^{-1}(0), \cE^1_P=\{\phi\in \cE^1: \phi\; \text{is invariant under the action of}\; P\}, \cE^1_{0, P}=\cE^1_P\cap \mathbb{I}^{-1}(0).\]
We note that $\text{Aut}_0(M)$ acts on the K\"ahler forms by pullback, for $\sigma\in \text{Aut}_0(M)$, $\sigma^{*}\omega_\phi=\omega_\psi$.
 Hence it induces an $\text{Aut}_0(M)$ action on $\cH_0$ (instead of $\cH$, since we need to specify the choice of the constant;  we also use the notation $\psi=\sigma[\phi]$ to indicate the action). The $\text{Aut}_0(M)$ action on $\cH_0$ is $d_1$-isometric and it extends to its metric completion $(\cE^1_0, d_1)$, see \cite{DR}[Lemma 5.10]. Note that $G$-action preserves $\cH_X^0$, and it extends to its metric completion $(\cE^1_{0, X}, d_1)$.

\subsection{The modified $\cK$-energy, its convexity and the properness}
We recall several well-known functionals  in K\"ahler geometry. Since notations are not used uniformly in literature, we include some details and our notations are mostly consistent with \cite{PG}[Chapter 4].  Fix a base metric $\omega$. We recall the $\mathbb{I}$-functional defined above,
\[
\mathbb{I}_\omega(\phi)=\frac{1}{(n+1)!}\int_M \phi \sum_{k=0}^n \omega^k\wedge \omega^{n-k}_\phi.
\]
When the base metric $\omega$ is changed to $\omega_\psi$, we write $\omega_\phi=\omega_\psi+\sqrt{-1}\p\bar\p (\phi-\psi)$ and we have
\[
\mathbb{I}_{\omega_\psi}(\phi)=\frac{1}{(n+1)!}\sum_{k=0}^n\int_M (\phi-\psi) \omega^k_{\psi}\wedge \omega_\phi^{n-k}
\]
For a given real $(1, 1)$ form $\chi$, we also define 
\[
\mathbb{I_{\omega, \chi}}(\phi)=\frac{1}{n!}\int_M \phi \sum_{k=0}^{n-1}\chi\wedge \omega^k\wedge \omega_\phi^{n-1-k}.
\]
We recall the so-called $\mathbb{J}$-functional, with the base metric $\omega$,
\begin{equation}\label{j0}
\mathbb{J}_\omega(\phi)=\sum_{k=0}^n\frac{1}{(n+1)!}\int_M \phi \omega^{n-k}\wedge\omega_\phi^k-\frac{1}{n!}\int_M \phi\omega^n_\phi.
\end{equation}
When the base metric $\omega$ is clear, we simply write $\mathbb{J}_\omega=\mathbb{J}$. It can be characterized by its derivative, 
\begin{equation}
\frac{d \mathbb{J}(\phi)}{d t}=\int_M \frac{\p \phi}{\p t} (\text{tr}_\phi \omega-n) \frac{\omega_\phi^n}{n!}
\end{equation}
We have the relation
\[
\mathbb{J}(\phi)=\mathbb{I}(\phi)-\frac{1}{n!}\int_M \phi\omega^n_\phi.
\]
Note that $\mathbb{I}_\omega(\phi)$ is a functional on $\phi\in \cH$, while $\mathbb{J}_\omega(\phi)$ does not depend on the normalization condition on $\phi$, hence a functional on K\"ahler metrics.
It is often written as
$\mathbb{J}_\omega(\omega_\phi)$. 
For a given real $(1, 1)$ form $\chi$, we also define $\mathbb{J}_{\omega, \chi}$ by
\begin{equation}\label{j1}
\begin{split}
\mathbb{J}_{\omega, \chi}(\phi)=&\frac{1}{n!}\int_M \phi \sum_{k=0}^{n-1}\chi\wedge \omega^k\wedge \omega_\phi^{n-1-k}-\frac{1}{(n+1)!}\int_M \underline{\chi} \phi \sum_{k=0}^n \omega^k\wedge \omega_\phi^{n-k}\\
=&\mathbb{I}_\chi(\phi)-\underline{\chi} \mathbb{I}(\phi)
\end{split}
\end{equation}
where $\underline \chi$ takes the form
\[
\underline \chi=\text{Vol}^{-1}(M, [\omega]) \int_M \chi\wedge \frac{\omega^{n-1}}{(n-1)!}. 
\]
We also define Aubin's $I$-functional and $J$-functional  as follows,
\begin{equation}\label{ij1}
\begin{split}
I_\omega(\phi):=I(\omega, \omega_\phi)=\frac{1}{n!}\int_M \phi\left(\omega^n-\omega_\phi^n\right)\\
J_\omega(\phi):=J(\omega, \omega_\phi)=\frac{1}{n!}\int_M\phi \omega^n-\mathbb{I}_\omega(\phi). 
\end{split}
\end{equation}
We have the following well-known facts, see for example \cite{PG}[Proposition 4.2.1],
\begin{equation}\label{ij2}
0\leq \frac{1}{n+1}I(\phi)\leq J(\phi)\leq \frac{n}{n+1}I(\phi)
\end{equation}
It follows that,
\begin{equation}\label{ij3}
\frac{1}{n+1} I(\phi)\leq \mathbb{J}(\phi)=I(\phi)-J(\phi)\leq \frac{n}{n+1}I(\phi)
\end{equation}
Now we recall that the Mabuchi's $\cK$-energy can be characterized by its variation,
\[
\delta K=-\int_M \delta \phi (R_\phi-\underline R)\frac{\omega_\phi^n}{n!},
\]
where $\underline{R}$ is the average of the scalar curvature.
Following Chen \cite{chen01}, the $\cK$-energy reads,
\begin{equation}\label{kenergy0}
\cK_\omega(\phi)=H_\omega(\phi)+\mathbb{J}_{\omega, -Ric} (\phi),
\end{equation}
where $H_\omega(\phi)$ is the entropy with the form
\[
H_\omega(\phi)=\int_M \log \left(\frac{\omega_\phi^n}{\omega^n}\right) \frac{\omega_\phi^n}{n!}.
\]
In particular, the $\mathbb{J}_{-Ric}$-functional takes the form
\[
\begin{split}
\mathbb{J}_{-Ric}(\phi)=&\frac{n\underline{R}}{(n+1)!}\int_M \phi \sum_{k=0}^n \omega^k\wedge \omega_\phi^{n-k}-\frac{1}{n!}\int_M \phi \sum_{k=0}^{n-1}Ric\wedge \omega^k\wedge \omega_\phi^{n-1-k}\\
=& \underline{R} \mathbb{I}(\phi)-\mathbb{I}_{Ric}(\phi)
\end{split}
\]

Now we recall the definition of the modified $\cK$-energy \cite{Guan, Simanca}.
For metrics in $\cH_X$,  the modified $\cK_X$ energy is characterized by
\begin{equation}\label{kenergy1}
\delta \cK_X=-\int_M \delta \phi (R_\phi-\underline R-\theta_X(\phi)) \frac{\omega_\phi^n}{n!}
\end{equation}
Hence the modified $\cK$-energy $\cK_X$ takes the form
\begin{equation}
\cK_X(\phi)=H(\phi)+\mathbb{J}_{-Ric}(\phi)+\mathbb{J}^X(\phi),
\end{equation}
where the functional $\mathbb{J}^X(\phi)$ takes the form, for a path $\phi_t$ connecting $0$ and $\phi$ in $\cH_X$, 
\[
\mathbb{J}^X(\phi)=\int_0^1 dt \int_M \dot\phi_t \theta_X(\phi_t)\frac{\omega^n_{\phi_t}}{n!}
\]
An important fact about the $\cK$-energy  ($\cK_X$) is that they are convex along $C^{1, 1}$ geodesics connecting points in $\cH$ (conjectured by Chen), proved by Berman-Berndtson \cite{BB} (see also  Chen-Li-Paun \cite{CLP}) 
\begin{thm}[Berman-Berndtsson]The $\cK$-energy $\cK(\phi_t)$  is convex in $t$ along the $C^{1, 1}$ geodesics $\phi_t$ connecting $\phi_0, \phi_1\in \cH$, for $t\in [0, 1]$. 
\end{thm}
Since $\mathbb{J}^X(\phi)$ is linear along a $C^{1, 1}$-geodesic, hence the modified $\cK$-energy $\cK_X$ is convex along the $C^{1, 1}$ geodesics in $\cH_X$ (see \cite{BB}[Section 4]). 
As a direct consequence of the convexity, 
\begin{cor}Suppose an extremal metric exists in $(M, [\omega], J)$ with extremal vector field $V$. Then an extremal metric minimizes the modified $\cK$-energy $\cK_X$ over $\cH_X$. 
\end{cor}

In \cite{BDL}, Berman-Davars-Lu extended the $\cK$-energy to $(\cE^1, d_1)$ and proved that the $\cK$-energy is convex along the geodesics in $\cE^1$. 
\begin{thm}[Berman-Davars-Lu]\label{bdl}The $\cK$-energy can be extended to a functional $\cK: \cE^1\rightarrow \R\cup\{+\infty\}$ using the formula \eqref{kenergy0}.
Such a $\cK$-energy in $\cE^1$ is the greatest $d_1$-lsc extension of $\cK$-energy on $\cH$. Moreover, $\cK$-energy is convex along the finite energy geodesics of $\cE^1$. 
\end{thm}
Let $\cE^1_X$ be the metric completion of $\cH_X$ in $\cE^1$ with respect to $d_1$ (see \cite{DR}[Lemma 5.4]). 
We can then extend $\cK_X$ to $\cE^1_X$ such that it is convex along the geodesics in $\cE^1_X$. 
First we recall a classical result of X.H. Zhu \cite{Zhu}[Corollary 5.3], 
\begin{prop}[X.H. Zhu]\label{z1}There exists a uniform constant $C_1$ (independent of $\phi\in \cH_X$) such that
$|\theta_X(\phi)|\leq C_1$.
\end{prop}

As a consequence, we have
\begin{prop}The functional $\mathbb{J}^X(\phi)$ on $\cH_X$ has a unique $d_1$-continuous extension to $\cE^1_X$.
\end{prop}
\begin{proof}
By definition of $\mathbb{J}^X$, if $\phi_t$ is any smooth path in $\cH_X$ with end point $\phi_0, \phi_1$
\[
\begin{split}
|\mathbb{J}^X(\phi_0)-\mathbb{J}^X(\phi_1)|&\leq \int_0^1dt\int_M |\dot \phi| |\theta_X(\phi_t)|\omega^n_{\phi_t}\\
&\leq C_1\int_0^1 dt\int_M |\dot \phi|\omega^n_{\phi_t}\leq C_1 l_1(\gamma),
\end{split}
\]
where $l_1(\gamma)$ stands for the length of the curve $\gamma$ with respect to $d_1$. Taking infimum over all such $\gamma$ we have
\[
|\mathbb{J}^X(\phi_0)-\mathbb{J}^X(\phi_1)|\leq Cd_1(\phi_0, \phi_1). 
\]
\end{proof}
As a direct consequence of Theorem \ref{bdl} and the linearity of $\mathbb{J}^X$ along geodesics in $\cE^1_X$, we have the following convexity of the modified $\cK$-energy. 
\begin{cor}The $\cK_X$-energy can be extended to a functional $\cK_X: \cE^1_X\rightarrow \R\cup\{+\infty\}$ using the formula \eqref{kenergy0}.
Thus extended, the $\cK_X$-energy in $\cE^1_X$ is the greatest $d_1$-lsc extension of $\cK_X$-energy on $\cH_X$ and $\cK_X$-energy is convex along the finite energy geodesics of $\cE^1_X$. Moreover, if an extremal metric exists with extremal vector field $V$, then it minimizes $\cK_X$ over $\cE^1_X$.  
\end{cor}

We need a notion of properness of the modified $\cK$-energy in terms of $d_1$  relative to $G$-action as follows. For any given metric $\omega_0$, we consider its $G$-orbit
\[
\cO_{\omega_0}=\{\phi\in \cH_0| \sigma^*\omega_0=\omega_\phi, \;\sigma\in G\}
\] 
Given any other K\"ahler metric $\omega_\varphi$ such that $\varphi\in \cH_0$,
the distance function \[d_1(\varphi, \cO_{\omega_0})=\inf_{\omega_\psi\in \cO_{\omega_0}} d_1(\varphi, \psi)\]
Note that since $G$ acts on $\cH_0$ isometrically, we know that
\[
\inf_{\omega_\psi\in \cO_{\omega_0}} d_1(\varphi, \psi)=\inf_{\sigma\in G} d_1(\varphi, \sigma[\psi])=\inf_{\sigma_1, \sigma_2\in G} d_1(\sigma_1 [\varphi], \sigma_2[\psi]). 
\]
Hence we define the $d_1$ distance relative to the action of $G$ as \[d_{1, G}(\varphi, \phi)=\inf_{\sigma\in G} d_1(\varphi, \sigma[\psi]),\]
where $\sigma[\psi]\in \cH_0$ denotes the K\"ahler potential of $\sigma^*\omega_\psi$.
We can restrict the distance $d_{1, G}$ to $\cH^0_X$ since $G$-action preserves $\cH^0_X$ and $\cH^0_X$ is a totally geodesic submanifold of $\cH_0$. 
\begin{defn}[Properness modulo  $G$]\label{rp}The modified $\cK$-energy is proper with respect to $d_{1, G}$ if the following conditions hold, 
\begin{enumerate}
\item The modified $\cK$-energy $\cK_X$ is bounded from below on $\cH_X$.
\item For any sequence $\phi_i\in \cH^0_X$, $d_{1, G}(0, \phi_i)\rightarrow \infty$ implies $\cK_X(\phi_i)\rightarrow \infty$. 
\end{enumerate}
\end{defn}

\begin{rmk}The notion of properness of $\cK$-energy was first introduced by G. Tian \cite{Tian94}, with the properness of $\cK$-energy in terms of Aubin's $J$-functional. Tian \cite{Tian97} proved that the properness of $\cK$-energy is equivalent to existence of K\"ahler-Einstein metric on Fano manifolds.  Later on Tian \cite{Tian01} has made a properness conjecture. Chen has made a properness conjecture with the properness of $\cK$-energy in terms of Mabuchi's metric $d_2$. A general properness conjecture was then clarified by Darvas-Rubinstein \cite{DR} in terms of $d_1$ and also Aubin's $J$-functional, modulo the action of $\text{Aut}_0(M)$. Our definition is  adaption of \cite{CC3}[Definition 3.1] to the extremal case. The notion of properness of modified $\cK$-energy has been studied extensively \cite{ZZ1, ZZ, LS1}.
\end{rmk}

As we shall see that the properness of $\cK_X$ with respect to $d_{1, G}$ is equivalent to the existence of a smooth extremal metric. It is often much more effective to check properness for metrics which are sufficiently symmetric. Let $K\subset \text{Aut}_0(M, V)$ be a maximal compact subgroup and we consider $K$-invariant metrics. We assume $\mathfrak{h}^0$, the Lie algebra of $\text{Aut}_0(M, V)$, satisfies the decomposition as in \eqref{calabi002}.  
We define the \emph{reductive properness} as follows,
\begin{defn}[Reductive properness]\label{rp1}
Given a maximal compact subgroup $K\subset\text{Aut}_0(M, V)$ such that \eqref{calabi002} holds for a $K$-invariant metric, 
we call the  $\cK_X$-energy  is reductive-proper, restricted on $\cH_K$, with respect to $d_{1, G}$ if the following conditions hold,
\begin{enumerate}
\item $\cK_X$ is bounded below on $\cH_{K}$.
\item For a sequence $\phi_i\in \cH_{K}^0$, then $d_{1, G}(0, \phi_i)\rightarrow \infty$ implies that $\cK_X(\phi_i)\rightarrow \infty$. 
\end{enumerate} 
\end{defn}
Note that in the definition above, we only assume that $\cK_X$ is bounded below on $\cH_K$, for $K$-invariant metrics. 
We shall need the following well-known fact \cite{Fut, M2, FM},
\begin{lemma}\label{ginvariant01}
The $\cK$-energy is invariant under the action of $\text{Aut}_0(M)$ if and only if the Futaki invariant is zero. If the $\cK_X$-energy is bounded below, then it is invariant under the action of $G=\text{Aut}_0(M, V)$. In general, the $\cK_X$-energy, when restricted on $\cH_K$, is invariant under the action of $G$ if \eqref{calabi002} is assumed. 
\end{lemma}

\begin{proof}
Suppose $\sigma\in \text{Aut}_0(M)$. Let $Y$ be a holomorphic vector field such that $\sigma_t$ is one-parameter subgroup generated by the flow of $Y_\R$, with $\sigma_0=id, \sigma_1=\sigma$. We write with $\sigma_t^*\omega=\omega+\sqrt{-1}\p\bar \p \phi(t)$.
By definition of $\cK$-energy, we have
\[
\frac{d}{dt}\cK(\sigma_t^*\omega)=-\int_M \sigma_{t}^*\left(\frac{d\phi}{dt}|_{t=0}(R_\omega-\underline R)\frac{\omega^n}{n!}\right)=-\int_M \frac{d\phi}{dt}|_{t=0}(R_\omega-\underline R)\frac{\omega^n}{n!}
\]
The righthand side is a constant (independent of $t$) and it is indeed just the real part of the Futaki invariant \cite{Fut} (hence a constant independent of $\omega$). 

Now we consider the $\cK_X$-energy. We consider the group action by $\text{Aut}_0(M, V)$, which preserve $V$. For $\sigma\in \text{Aux}_0(M, V)$, let $\sigma_t$ be one parameter subgroup generated by the flow of $Y_\R$.  We know that $\sigma_t^*\omega$ in invariant with respect to $JV$. Hence $\cK_X(\sigma_t^*\omega)$ is well defined and real-valued. We compute, by definition of modified Mabuchi's energy, 
\begin{equation}
\frac{d}{dt}\cK_X(\sigma_t^*\omega)=-\int_M \sigma_t^{*}\left(\frac{d \phi}{dt}|_{t=0}(R_{\omega}-\underline R-\theta_X)\frac{\omega^n}{n!}\right)
\end{equation}
We compute \[
\frac{d}{dt}\sigma_t^* \omega|_{t=0}= L_{Y_\R}\omega=\sqrt{-1}\p\bar \p\text{Re}(\theta_Y(\omega)),
\]
where $\theta_Y(\omega)$ is the complex potential of $Y$ with respect to $\omega$. 
Hence we have
\begin{equation}\label{reduce01}
\frac{d}{dt}\cK_X(\sigma_t^*\omega)=-\int_M \text{Re}(\theta_Y)(R_{\omega}-\underline R)\frac{\omega^n}{n!}+\int_M \text{Re}(\theta_Y) \theta_X\frac{\omega^n}{n!}
\end{equation}
The righthand side is independent of $t$. Hence if $\cK_X$ is bounded from below over $\cH_X$, the righthand side has to be zero. This implies that $\cK_X$ is invariant under the action of $\text{Aut}_0(M, V)$. In most general case, suppose we do not know $\cK_X$ is bounded below a priori ($\sigma_t^*\omega$ might not be $K$-invariant in general). We consider $K$-invariant metrics. By \eqref{calabi002}, we only need to consider $Y_\R\in J\mathfrak{ham}$. We still have \eqref{reduce01}; note that the righthand side is independent of the choice of $\omega$ (this conclusion does not require $\omega$ to be $K$-invariant) for any $Y\in \mathfrak{ham}\oplus J\mathfrak{ham}$. The first term is the real part of Futaki invariant, the second term is Futaki-Mabuchi bilinear form. Both are independent of the choice of the metric $\omega$. By the definition of extremal vector field, the righthand side of \eqref{reduce01} is always zero \cite{FM}. 

\end{proof}

It might be more convenient to consider a maximal subgroup $K_0\subset \text{Aut}_0(M, V)$ such that its Lie algebra consists of holomorphic vector fields with nonempty zeros. Let $G_0$ be the complexification of $K_0$. Note that $G_0$ is simply the reduced part of $\text{Aut}_0(M, V)$.  The Lie algebras satisfy $\mathfrak{g}^0=\mathfrak{k}_0\oplus J\mathfrak{k}_0$ (when an extremal metric exists, with extremal vector field $V$, the Lie algebra corresponds to $\mathfrak{ham}\oplus J\mathfrak{ham}$). Using the pair $(K_0, G_0)$, we can define the reduced properness:
\begin{defn}[reduced properness]\label{rp2}We say $\cK_X$ is reduced-proper if the following conditions hold
\begin{enumerate}
\item $\cK_X$ is bounded below on $\cH_{K_0}$.
\item $\cK_X$, when restricted on $\cH_{K_0}$, is proper with respect to $d_{1, G_0}$. That is, for a sequence of $\phi_i\in \cH^0_{K_0}$, then $d_{1, G_0}(0, \phi_i)\rightarrow \infty$ implies that $\cK_{X}(\phi_i)\rightarrow \infty.$
\end{enumerate}
\end{defn}

The following is well-known \cite{Fut, M2, FM},
\begin{lemma}\label{fm02}The $\cK_X$-energy is $G_0$ invariant. 
\end{lemma}
\begin{proof}This is simply the definition of the extremal vector field, when a maximal compact group $K_0$ specified,  via the Futaki-Mabuchi bilinear form \cite{FM}. 
\end{proof}

\section{Extremal metrics and properness}

We prove Theorem \ref{main1} and Theorem \ref{main2} in this section, which we recall as follows,
\begin{thm}Let $(M, [\omega, J])$ be a compact K\"ahler manifold with extremal vector field $V$. Then the following statements are equivalent
\begin{enumerate}
\item $(M, [\omega, J])$ admits an extremal metric with extremal vector field $V$.

\item The $\cK_X$-energy is proper with respect to $d_{1, G}$. 

\item The $\cK_X$-energy, restricted on $\cH_K$, is reductive-proper with respect to $d_{1, G}$.

\item The $\cK_X$-energy, restricted on $\cH_{K_0}$, is reduced-proper with respect to $d_{1, G_0}$

\end{enumerate}
\end{thm}

\subsection{Chen-Cheng's a priori estimates}
We summarize Chen-Cheng's a prior estimates regarding csck type equations. Consider the following equation for $\phi\in \cH$, 
\begin{equation}
\begin{split}&\frac{\det(g_{i\bar j}+\phi_{i\bar j})}{\det(g_{i\bar j})}=e^F\\
&\Delta_\phi F=f+\text{tr}_\phi \eta,
\end{split}
\end{equation}
where $f$ is a smooth function and $\eta$ is a smooth real $(1, 1)$ form. 
\begin{thm}[Chen-Cheng, Theorem 3.1 \cite{CC2}]\label{keyestimate1}There exists a uniformly bounded positive constant $C_0\geq 1$, depending only on the background geometry $(M, \omega)$, the upper bound of entropy $\int_M \log(\frac{\omega_\phi^n}{\omega^n}) \omega^n_\phi$, and $\max |f|$, $\max |\eta|_{\omega}$ such that
\begin{equation}
\|\nabla \phi\|_{C_0}\leq C_0,  C_0^{-1} \omega\leq \omega_\phi\leq C_0\omega. 
\end{equation}
\end{thm}
Consider also the following more general equation for $\phi$,
\begin{equation}\label{continuity8}
\begin{split}&\frac{\det(g_{i\bar j}+\phi_{i\bar j})}{\det(g_{i\bar j})}=e^F\\
&\Delta_\phi F=R+\text{tr}_\phi (Ric-\beta),
\end{split}
\end{equation}
where $\beta=\beta_0+\sqrt{-1}\p\bar \p f\geq 0$. Assume $\beta_0$ is a bounded $(1, 1)$-form and $f$ is normalized such that $\sup f=0, e^{-f}\in L^{p_0}(M)$ for some $p_0>1$ specified below.
\begin{thm}[Chen-Cheng, Theorem 2.3 \cite{CC3}]\label{keyestimate2}Assume $\beta\geq 0$. Let $\phi$ be a smooth solution of \eqref{continuity8}. Suppose $p_0\geq k_n$ for some positive constant $k_n$ depending only on $n:=\text{dim}_\C M$  ($k_n$ sufficiently large, say $10n^2+100$) . Then for any $1\leq p\leq p_0-1$,
\begin{equation}
\|F+f\|_{W^{1, 2p}}+\|n+\Delta\phi\|_{L^p}\leq C,
\end{equation}
where $C$ depend on an upper bound of entropy $\int_M \log(\frac{\omega_\phi^n}{\omega^n}) \omega^n_\phi$, $p_0$, and the bound of $\int_M e^{-p_0f} \omega^n$, $\max |R|, \max|\beta_0|_{\omega}$ and the background metric $(M, \omega)$. The dependence is uniform in $p_0$. 
\end{thm}

\begin{rmk}Technically, the arguments by Chen and Cheng are subtle. The key is to consider the scalar curvature type equation as two coupled elliptic equations  on potential $\phi$ and log-volume form $F$.
Their methods are combination of very skillful applications of maximum principle  and extremely skillful integral methods on compact manifolds. The results and ideas using integral method on compact manifolds in \cite{chenhe12} are actually used in a significant way in \cite{CC1, CC2, CC3}. More strikingly, Chen and Cheng have obtained $C^0$ estimates (and all higher order estimates)  depending only on the entropy bound (and  scalar curvature bound of course). As a comparison, the dependence on volume-ratio in this case is even weaker, than the $C^0$ estimate obtained by Kolodziej \cite{K} for the complex Monge-Ampere equations. This dependence down to the entropy and scalar curvature is absolutely crucial for geometric applications such as properness conjecture. 
\end{rmk}

\subsection{Properness implies existence}In this section we prove the existence of an extremal metric given the properness of the modified $\cK$-energy modulo $\text{Aut}_0(M, V)$. we shall  see that the proof can be modified slightly for reductive-properness and reduced properness.   We  consider the modified continuity path of Chen \cite{chen15} for extremal metrics, 
\begin{equation}\label{continuity1}
t(R_\phi-\underline R-\theta_X(\phi))-(1-t) (tr_\phi \omega-n)=0
\end{equation}
where $\omega, \omega_\phi$ are both $T$-invariant. 
The functional corresponding the continuity path \eqref{continuity1} is denoted by
\begin{equation}
\tilde K_{t}=t\cK_X(\phi)+(1-t)\mathbb{J}(\phi).
\end{equation}

\subsubsection{Openness} First we prove the openness of the modified continuity path \eqref{continuity1}.
The openness when $V=0$ follows from the results of \cite{chen15, hashimoto, yuzeng}. We generalize their results to extremal case. 
Technically the following theorem is a generalization of  Hashimoto's result \cite{hashimoto}[Theorem 1.2] for $V=0$, and its proof is a modification of Hashimoto's proof. 
\begin{thm}\label{open1}Suppose we have two K\"ahler metrics $\omega$ and $\alpha$ such that $\Lambda_\omega \alpha=\text{const}. $ We assume $\omega, \alpha$ are both $T$-invariant.  Then there exists a constant $r(\omega, \alpha)$ depending only on $\omega, \alpha$ such that $r\geq r(\omega, \alpha)$, there exists $\phi\in \cH_X$ such that $\omega_\phi$ satisfies $R_\phi-\theta_X(\phi)-\Lambda_{\omega_\phi}(r\alpha)=\text{const}$.\end{thm}

\begin{proof}We call $\omega$ a $\alpha$-twisted extremal metric if it satisfies the equation
\[
R(\omega)-\theta_X-\Lambda_{\omega}\alpha=\text{const}
\]
Denote $S(\omega)=R(\omega)-\theta_X-\Lambda_{\omega}\alpha$. We compute its linearization, given $\delta \omega=\sqrt{-1}\p\bar\p\psi$, 
\begin{equation}\label{l1}
\cL_S(\psi)=-\cD^*_\omega\cD_\omega \psi+(\p S(\omega), \bar\p \psi)_\omega+(\alpha, \sqrt{-1}\p\bar\p \psi)_\omega+(\p \Lambda_\omega\alpha, \bar\p\psi)_\omega,
\end{equation}
where $\cD_\omega: C^\infty(X)\rightarrow C^\infty(T^{1, 0}M\otimes \Omega^{0, 1})$ is defined by $\cD_\omega \psi=\bar\p (\nabla^{1, 0}\psi)$, and $\cD_\omega^*$ is its formal adjoint and $\cD^*\cD$ is called Lichnerowicz operator.  The linearization of scalar curvature was first computed in Calabi's paper \cite{Ca1} and was studied extensively by LeBrun-Simanca \cite{LS} when they considered deformation of extremal metrics. 
Our formula \eqref{l1} is a straightforward modification of Hashimoto's computation \cite{hashimoto}[Page 5]. Note that we have $\theta_X(\phi)=\theta_X+X(\phi)$, for $\omega_\phi=\omega+\sqrt{-1}\p\bar\p\phi$, hence  the extra term is given by $\delta \theta_X=X(\psi)=(\p \theta_X, \bar\p \psi)_\omega$ ($\theta_X$ is real for $T$-invariant metrics). 
This together with Hashimoto's computation we have \eqref{l1}. At an $\alpha$-twisted extremal metric, we note that the linearized operator reads
\begin{equation}\label{l2}
\cL_S(\psi)=-\cD^*_\omega\cD_\omega \psi+F_{\omega, \alpha},
\end{equation}
where we define the operator, following \cite{hashimoto}[Section 2],  $F_{\omega, \alpha}: C^\infty\rightarrow C^\infty$ by
\begin{equation}\label{of}
F_{\omega, \alpha}(\psi):=(\alpha, \sqrt{-1}\p\bar\p \psi)_\omega+(\p \Lambda_\omega\alpha, \bar\p\psi)_\omega. 
\end{equation}
By \cite{hashimoto}[Lemma 2.2] we know that $F_{\omega, \alpha}$ is a complex self-adjoint operator and its kernel consists of constant functions.  The strategy is,  first we construct approximate solutions to $r\alpha$ for $r>0$ sufficiently large, and then apply the inverse function theorem to get an actual solution.
To construct approximation, we observe the following,
\begin{prop}Let $\omega, \alpha$ be $T$-invariant K\"ahler metrics satisfying  $\Lambda_\omega\alpha=\text{const}$ and let $\rho(\omega, V)$ be the unique solution of the equation 
\[
\Delta_\omega \rho=R(\omega)-\underline R-\theta_X, \int_M \rho \omega^n=0. 
\]
Let  $\alpha_1=\alpha+\frac{\sqrt{-1}}{r} \p\bar \p \rho$ be K\"ahler for $r>0$ sufficiently large, depending only on $\omega, \alpha$. Then $\omega$ is $r\alpha_1$-twisted extremal metric. Note that $\rho$ and hence $\alpha_1$ are both $T$-invariant. 
\end{prop}
Thus $\omega$ is an $r\alpha_1$-twisted extremal metric  for $\alpha_1\in [\alpha]$ which differs $\alpha$ by order $r^{-1}.$ Next we want to improve this observation order by order, so that $\omega_m\in [\omega]$ is an $r\alpha_m$-twisted extremal metric for $\alpha_m\in [\alpha]$ which differs $\alpha$ by order $r^{-m}$. We proceed as follows. Suppose $\Lambda_\omega\alpha=\text{const}$. We write
\[
R(\omega)-\theta_X-r\Lambda_\omega\alpha=\text{const}+(R(\omega)-\underline R-\theta_X).
\]
Consider $\omega_1:=\omega+r^{-1}\p\bar \p \phi_1$. We write $\theta_X(\omega_1)=\theta_X+X(r^{-1}\phi_1)$. We compute,
\begin{equation}\label{step1}
R(\omega_1)-\theta_X(\omega_1)-r\Lambda_{\omega_1}\alpha=\text{const}+(R(\omega)-\underline R-
\theta_X)+(\alpha, \sqrt{-1}\p\bar\p \phi_1)_\omega+O(r^{-1}),
\end{equation}
where we have used the fact that $|X(\phi_1)|$ is uniformly bounded by Proposition \ref{z1}.
Given $\Lambda_\omega\alpha=\text{const}$, we use \cite{hashimoto}[Lemma 3.2] to solve the linear elliptic equation of $\phi_1$ such that \begin{equation}\label{e1}(R(\omega)-\underline R-
\theta_X)+(\alpha, \sqrt{-1}\p\bar\p \phi_1)_\omega=0.\end{equation} The point is that the operator $F_{\omega, \alpha}$ in \eqref{of} is a second order linear elliptic operator which has zero kernel modulo constants. 
We emphasize two properties of the solution $\phi_1$. The first is that $\phi_1$ is $T$-invariant, and the second is that we have uniform control on derivatives of  $\phi_1$ in terms of $\omega, \alpha$ by the elliptic equation \eqref{e1}. With this in mind we rewrite \eqref{step1} as
\begin{equation}
\label{step2}
R(\omega_1)-\theta_X(\omega_1)-r\Lambda_{\omega_1}\alpha=\text{const}+r^{-1} f_{1, r},
\end{equation}
where $f_{1, r}$ is a uniformly bounded smooth function. In particular we have $\omega_1=\omega+O(r^{-1})$. 
We repeat the procedure (by induction) as follows. Suppose we have constructed $\phi_1, \cdots, \phi_{m-1}$ and correspondingly for $i=1, \cdots, m-1$,
 \[\omega_i:=\omega+\sqrt{-1}\p\bar\p (r^{-1}\phi_1+\cdots+r^{-i}\phi_i)\] 
 such that for $1\leq i\leq m-1$, we have for a uniformly bounded smooth function $f_{i, r}$ (we can assume $\int_Mf_{i, r}\omega_i^n=0$),
 \[
R(\omega_i)-\theta_X(\omega_i)-r\Lambda_{\omega_i}\alpha=\text{const}+r^{-i} f_{i, r}.
\]
Next we construct $\phi_m$ and $\omega_m:=\omega_{m-1}+r^{-m}\phi_m$ as follows. We write
\[
\begin{split}
R(\omega_m)-\theta_X(\omega_m)-r\Lambda_{\omega_m}\alpha=&(R(\omega_{m-1})-\theta_X(\omega_{m-1})-r\Lambda_{\omega_{m-1}}\alpha)\\&\;+r^{1-m}(\alpha, \sqrt{-1}\p\bar\p \phi_m)_{\omega_{m-1}}+O(r^{-m})\\
=&\text{const}+r^{1-m}f_{m-1, r}+r^{1-m}(\alpha, \sqrt{-1}\p\bar\p \phi_m)_{\omega}+O(r^{-m}),
\end{split}
\]
where we have used the fact that $\omega_{m-1}=\omega+O(r^{-1})$.  We can then solve $\phi_m$ such that
\[
(\alpha, \sqrt{-1}\p\bar\p \phi_m)_{\omega}+f_{m-1, r}-\underline{f_{m-1}, r}=0,
\]
where $\underline{f_{m-1, r}}$ is the average of $f_{m-1, r}$ with respect to $\omega^n$. 
The standard elliptic theory then gives a uniform control on $\phi_m$. We summarize the discussions above as, 

\begin{lemma}Suppose $\omega, \alpha$ are $T$-invariant and $\Lambda_\omega\alpha=\text{const}$. Then for each $m\in \N$, there exist $\phi_1, \cdots, \phi_m\in C^\infty$ such that
\[
\omega_m=\omega+\sqrt{-1}\p\bar\p (r^{-1}\phi_1+\cdots+r^{-m}\phi_m)
\]
such that for a function $f_{m, r}$ with $\int_M f_{m, r}\omega_m^n=0$ (which is uniformly bound in $C^\infty$ for all sufficiently large $r$),
\[
R(\omega_m)-\theta_X(\omega_m)-r\Lambda_{\omega_m}\alpha=\text{const}+r^{-m} f_{m, r}.
\]
\end{lemma}
We can then solve the equation $\Delta_{\omega_m} h_{m, r}=f_{m, r}$ such that $\int_M h_{m, r}\omega_m^n=0$. As noted above, $\omega_m=\omega+O(r^{-1})$ is uniformly equivalent to $\omega$ for each $m$ (independent of $r$ and $m$ when $r$ is sufficiently large). The standard elliptic theory then gives uniform control on $h_{m, r}$. Denote $\alpha_m=\alpha+r^{-m}\sqrt{-1}\p\bar\p h_{m, r}$. Then it satisfies the equation,
\begin{equation}\label{approx1}
R(\omega_m)-\theta_{X}(\omega_m)-r\Lambda_{\omega_m}\alpha_m=\text{const}. 
\end{equation}
Hence $\alpha_m-\alpha=O(r^{-m})$ and we have obtained $r\alpha_m$-twisted extremal metric $\omega_m$. 
We emphasize that the assumption that $\omega, \alpha$ are both $T$-invariant is necessary, which are used implicitly that all solutions $\phi_i$ are $T$-invariant, hence $\theta_X(\omega_i)$ are all real valued.
Next we apply the inverse function theorem between Banach spaces to perturb $r\alpha_m$-twisted extremal metric $\omega_m$ to get an $r\alpha$-twisted extremal metric. The argument proceeds exactly the same as in \cite{hashimoto}[Section 3.2] with two modifications. We replace the twisted csck equation  by the twisted extremal equation
\[
R(\omega)-\theta_X(\omega)- r\Lambda_\omega\alpha=\text{const},
\]
and hence  by \eqref{l1} and in particular \eqref{l2}, the linearized operator is exactly the same as in the twisted csck case (compare \cite{hashimoto}, (5)). We need to apply inverse function theorem in a $T$-invariant way, by considering functional spaces which are $T$-invariant.  We briefly summarize the argument. Consider $T$-invariant Sobolev spaces $W^{k, 2}_{0, T}$ for $k$ sufficiently large, normalized such that the average is zero with respect to $\omega$, and $\Omega^{1, 1}_T$ be $W^{k, 2}$ real $(1, 1)$-forms on $M$ which are invariant with respect to $T$-action. Take the Banach space $B_1:=\Omega^{1, 1}_T \times W^{k+4, 2}_{0, T}$, $U:=\{(\eta, \phi)\in B_1| \omega_m+\sqrt{-1}\p\bar\p \phi>0\}$. Take $B_2:=\Omega^{1, 1}_T \times W^{k, 2}_{0, T}$.
We define the map, for constant $r>0$, 
\[
T_r(\eta, \phi):=(\alpha_m+\eta, R(\omega_{m, \phi})-\theta_X(\omega_{m, \phi})-r\Lambda_{\omega_{m, \phi}}(\alpha_m+\eta))
\]
with $\omega_{m, \phi}=\omega_m+\sqrt{-1}\p\bar\p\phi$. We identify $(\alpha_m, \omega_m)$ to be $0\in B_1$. We also write the operators $\Lambda_m=\Lambda_{\omega_m}, \cD^*_m\cD_m$ and $F_m$ for corresponding operators defined using $\omega_m, \alpha_m$. \eqref{l2} gives the linearized operator $DT_r$ at $0\in B_1$ by ($\omega_m$ is $r\alpha_m$-twisted)
\[
DT_r|_0(\tilde \eta, \tilde \phi)=\begin{pmatrix}1 & 0\\
-r\Lambda_m& -\cD^*_m\cD_m+rF_m
\end{pmatrix}\begin{pmatrix}\tilde\eta\\
\tilde\phi\end{pmatrix}
\]
Since the kernel of $\cD^*_m\cD_m+rF_m: W^{k+4, 2}_0\rightarrow W^{k, 2}_0$ is zero, it follows that $DT_r|_0: B_1\rightarrow B_2$ is an invertible operator with the inverse $P=P_r$,
\[
P_r(\tilde \eta, \tilde \phi)=\begin{pmatrix}1 & 0\\
r(-\cD^*_m\cD_m+rF_m)^{-1}\Lambda_m& (-\cD^*_m\cD_m+rF_m)^{-1}
\end{pmatrix}\begin{pmatrix}\tilde\eta\\
\tilde\phi\end{pmatrix}
\] 
We have the estimate $\|P_r\|\leq Cr$ for uniform constant $C$ and $r>1$ as in \cite{hashimoto}. We compute $T_r(0, 0)=(\alpha_m, \underline R-rc)$, where $c$ is the topological constant  (the average of $\Lambda_{\omega_m}{\alpha_m}$, depending only on $[\omega]$ and $\alpha$). We need to find $\eta, \phi$ such that $T_r(\eta, \phi)=(\alpha, \underline R-rc)$. Then the exact same application of inverse function theorem applies for $\alpha-\alpha_m=O(r^{-m})$, $m$ sufficiently large (\cite{hashimoto}[Theorem 3.4, Section 3.2]).  

\end{proof}
Similar to \cite{hashimoto}[Corollary 1.4], we have the following,
\begin{cor}\label{open2}Suppose that $\omega$ is $\alpha$-twisted extremal for two $T$-invariant K\"ahler metrics $\alpha, \omega$. Then if $\tilde \alpha$ is a $T$-invariant metric such that $\tilde \alpha-\alpha$ is sufficiently small in the $C^\infty$-norm, then there exists an $\tilde \alpha$ twisted extremal metric in $[\omega]$ which is $T$-invariant.  
\end{cor}

\begin{rmk}Theorem \ref{open1} and Corollary \ref{open2} have straightforward generalizations to a $K$-invariant setting, if we assume that $\omega$ and $\alpha$ are both $K$-invariant, where $K$ is a (maximal) compact subgroup of $\text{Aut}_0(M, V)$. The only modification we need is to replace $T$ by $K$. 
\end{rmk}

\subsubsection{Closedness}Theorem \ref{open1} and Corollary \ref{open2} imply that we can solve the equation \eqref{continuity1} in an interval $[0, t_0)$ for some $0<t_0\leq 1$. 
Now we show that $t_0=1$ under the assumption that the modified $\cK$-energy $\cK_X$ is bounded below in $\cH_X$. 
Our argument is a modification of Chen-Cheng \cite{CC3}[Section 3]. 
Note that we can rewrite the equation as
\begin{equation}\label{continuity2}
\begin{split}
&\omega_\phi^n=e^F \omega^n\\
&\Delta_\phi F=-(\underline R-(1-t)n/t+\theta_X(\phi))+\text{tr}_\phi(Ric(\omega)-(1-t)\omega/t)
\end{split}
\end{equation}

Given the estimates above, we have a direct consequence that \eqref{continuity1} has a smooth solution for $t\in [0, 1)$. 
\begin{lemma}\label{kbelow1}Suppose the modified $\cK$-energy $\cK_X$ is bounded below on $\cH_X$, then \eqref{continuity1} has a smooth solution for $t\in [0, 1)$. 
\end{lemma}

\begin{proof}Since $\cK_X$ is bounded below, it follows that \[\tilde K_t\geq C_1(1-t)d_1(0, \phi)-C\] for any $t\in [0, 1)$. This depends on \eqref{ij3},
\[
\mathbb{J}(\phi)\geq  \frac{1}{n+1}I(\phi), 
\]
and the estimate in \cite{DR}[Proposition 5.5],
\[
I(\phi)\geq C^{-1}d_1(0, \phi)-C, \phi\in \cH_0.
\]
Suppose $\phi(t)\in \cH^0$ solves \eqref{continuity1}. By the convexity of the $\cK_X$ and  the convexity of $\tilde K_t$ (along the geodesics in $\cH_X$),  it follows that $\phi(t)$ minimizes $\tilde K_t$. Hence we have
$\tilde K_t<\infty$.    
It follows that for any $t\in [0, 1)$,
\[d_1(0, \phi)<C((1-t)^{-1}+1)\] 
Note that both $\mathbb{J}_{-Ric}(\phi)$ and $\mathbb{J}^X(\phi)$ are bounded by $Cd_1(0, \phi)$ for some uniformly bounded constant $C$. It follows that, for any $t\in [0, 1)$, 
\begin{equation}\label{t0}
n!H(\phi)=\int_M \log \frac{\omega_{\phi}^n}{\omega^n} \omega_{\phi}^n<C((1-t)^{-1}+1).
\end{equation}
It then follows from Theorem \ref{keyestimate1} that for any $t\in [0, 1)$, \eqref{continuity1} has a smooth solution. 
\end{proof}

Now we consider the behavior when $t\rightarrow 1$. The purpose is to show for any sequence $t_i\rightarrow 1$, there exists a limit of $\omega_{\tilde \phi_i}$ after modifying by suitable choice of elements in $G$, and the limit defines a smooth extremal metric. Note that the estimate in \eqref{t0}  blows up when $t\rightarrow 1$. We need to use the properness of modified $\cK$-energy in an effective way.  Since the properness only implies a distance bound of $d_{1, G}$, it is then necessary to apply an automorphism $\sigma_i$ in $G$ (at each time $t_i$) such that the resulting potential remains in a bounded set of $\cH_X$. We proceed as follows. 
Let $\tilde \phi_i\in \cH^0_X$ be the solution of \eqref{continuity1} at $t_i$, for $t_i$ increasing to $1$. The starting point is to show that $\tilde \phi_i$ is a minimizing sequence of the modified $\cK_X$ energy. 
\begin{lemma}\label{kenergy1}We have the following, 
\begin{equation}\label{j100}
\begin{split}
&\tilde K_{t_i}(\tilde \phi_i)=\inf_{\phi\in \cH_X} \tilde K_{t_i}(\phi)\rightarrow \inf_{\cH_X} \cK_X(\phi), t_i\rightarrow 1\\
&\cK_X(\tilde \phi_i)\rightarrow \inf_{\cH_X} \cK_X(\phi), t_i\rightarrow 1\\
&(1-t_i)\mathbb{J}(\tilde \phi_i)\rightarrow 0.
\end{split}
\end{equation}
\end{lemma} 
\begin{proof}
Since $\cK_X$ is bounded below over $\cH_X$, we can choose $\phi^\epsilon \in\cH_X$ such that $\cK_X(\phi^\epsilon)\leq \inf \cK_X+\epsilon$.
Note that $t\cK_X+(1-t)\mathbb{J}$ is convex and hence $\tilde \phi_i$ minimizes $\tilde K_{t_i}$. 
Hence
\[ \tilde K_{t_i}(\tilde \phi_i)\leq t_i\cK_X(\phi^\epsilon)+(1-t_i)\mathbb{J}(\phi^\epsilon).
\]
It follows that 
\[
\lim\sup_{i\rightarrow\infty} \tilde K_{t_i}(\tilde \phi_i)\leq \cK_X(\phi^\epsilon)\leq \inf_{\cH_X} \cK_X+\epsilon.
\]
On the other hand, since $\mathbb{J}\geq 0$, for $i$ sufficiently large,
\[
t_i \inf_{\cH_X} \cK_X\leq t_i\cK_X(\tilde \phi_i)\leq \tilde K_{t_i}(\tilde\phi_i)\leq \inf_{\cH_X} \cK_X+\epsilon
\]
This proves all three statements in \eqref{j100}. 
\end{proof}
As a direct consequence of properness with respect to $d_{1, G}$ and Lemma \ref{kenergy1}, this gives the desired distance bound modulo $G$.
\begin{cor}We have the  bound on $d_{1, G}$, 
\[
\sup_i d_{1, G}(0, \tilde \phi_i)<\infty. 
\]
\end{cor}
Hence we can find a $\sigma_i\in G, \phi_i\in \cH^0_X$ such that
\begin{equation}
\omega_{\phi_i}=\sigma_i^*{\omega_{\tilde \phi_i}},\;\text{and}\; \sup_{i} d_1(0, \phi_i)<\infty. 
\end{equation}
With the uniform bound of distance $d_1(0, \phi_i)$, we need to show that $\phi_i$ converges uniformly. 
We show that the entropy of $\phi_i$ is uniformly bounded in the next,
\begin{prop}\label{e1000}
We have
\begin{equation}\label{entropy100}
H_\infty:=\sup_in!H(\phi_i)=\sup_i \int_M \log \frac{\omega_{\phi_i}^n}{\omega^n} \omega_{\phi_i}^n<\infty
\end{equation}
\end{prop} 
\begin{proof}
By Lemma \ref{ginvariant01}, $\cK_X$ is invariant under the action of $\text{Aut}_0(M, V)$ since $\cK_X$ is bounded below over $\cH_X$. Hence  $\sup_i\cK_X(\phi_i)=\sup_i \cK_X(\tilde \phi_i)<\infty$. 
Recall 
\[
\cK_X(\phi)=H(\phi)+\mathbb{J}_{-Ric}(\phi)+\mathbb{J}^X(\phi)
\]
Since  $|\mathbb{J}_{-Ric}|, |\mathbb{J}^X|$ are bounded when $\sup_id_1(0, \phi_i)<\infty$. This implies that $H_\infty<\infty.$ \end{proof}

We need to consider the equation which $\phi_i$ satisfies. Denote $\omega_i=\sigma_i^* \omega=\omega+\sqrt{-1}\p\bar\p h_i$, with the normalization $\sup h_i=0$ (note that $h_i$ is in $\cH_X$, but not in $\cH^0_X$ in general). 

\begin{lemma}\label{l100}The potential $\phi_i$ satisfies the following equations
\begin{equation}\label{equation-i0}
\begin{split}
&\omega_{\phi_i}^n=e^{F_i}\omega^n\\
&\Delta_{\phi_i} F_i=(\underline R-\frac{1-t_i}{t_i}n+\theta_i)+\text{tr}_{\phi_i}(Ric(\omega)-\frac{1-t_i}{t_i}\omega_i),
\end{split}
\end{equation}
where $\theta_i=\theta_{X}(\tilde \phi_i)\circ \sigma_i=\theta_X(\phi_i)$. 
\end{lemma}

\begin{proof}We write $\tilde F_i=\log \frac{\omega^n_{\tilde \phi_i}}{\omega^n}$. We have the following,
\[
\omega_{\phi_i}^n=\sigma_i^{*} (\omega_{\tilde \phi_i}^n),\; \sigma_i^{*} (e^{\tilde F_i}\omega^n)=e^{\tilde F_i \circ \sigma_i} \omega_i^n
\]
It follows that $\omega_{\phi_i}^n=e^{\tilde F_i \circ \sigma_i} \omega_i^n.$ Denote $F_i=\log \frac{\omega^n_{\phi_i}}{\omega^n}$. Hence we have
\begin{equation}\label{volume1}
F_i=\tilde F_i \circ \sigma_i+\log \frac{\omega_i^n}{\omega^n}
\end{equation}
Now we write the second equation (at $t=t_i$) in \eqref{continuity1}  in its equivalent form,
\[
\sqrt{-1} \p\bar\p \tilde F_i\wedge \frac{\omega^{n-1}_{\tilde \phi_i}}{(n-1)!}=-(\underline R-\frac{1-t_i}{t_i}n-\theta_{X}(\tilde \phi_i))\frac{\omega^n_{\tilde \phi_i}}{n!}+(Ric(\omega)-\frac{1-t_i}{t_i}\omega)\wedge \frac{\omega^{n-1}_{\tilde \phi_i}}{(n-1)!}
\]
Pulling back by $\sigma_i$, we obtain
\[
\sqrt{-1} \p\bar\p (\tilde F_i\circ \sigma_i)\wedge \frac{\omega^{n-1}_{\phi_i}}{(n-1)!}=-(\underline R-\frac{1-t_i}{t_i}n-\theta_{X}(\tilde \phi_i)\circ \sigma_i)\frac{\omega^n_{\phi_i}}{n!}+(Ric(\omega_i)-\frac{1-t_i}{t_i}\omega_i)\wedge \frac{\omega^{n-1}_{\phi_i}}{(n-1)!}
\]
Combining \eqref{volume1}, we obtain
\begin{equation}\label{equation-i1}
\sqrt{-1} \p\bar\p F_i\wedge \frac{\omega^{n-1}_{\phi_i}}{(n-1)!}=-(\underline R-\frac{1-t_i}{t_i}n-\theta_{X}(\tilde \phi_i)\circ \sigma_i)\frac{\omega^n_{\phi_i}}{n!}+(Ric(\omega)-\frac{1-t_i}{t_i}\omega_i)\wedge \frac{\omega^{n-1}_{\phi_i}}{(n-1)!}
\end{equation}
For simplicity we denote \begin{equation}\label{potential1}\theta_{i}:=\theta_{X}(\tilde \phi_i)\circ \sigma_i=\theta_{(\sigma_i)^{-1}_*X}(\phi_i)=\theta_X(\phi_i),\end{equation} 
where we use that $G$-action preserves $X$. 
We can rewrite \eqref{equation-i1} as
\begin{equation}
\Delta_{\phi_i} F_i=(\underline R-\frac{1-t_i}{t_i}n+\theta_i)+\text{tr}_{\phi_i}(Ric(\omega)-\frac{1-t_i}{t_i}\omega_i)
\end{equation}
This completes the proof.
\end{proof}

With the preparation above, we can then state the main theorem which gives an extremal metric with extremal vector field $V.$
\begin{thm}\label{con100}When $t_i\rightarrow 1$, $\omega_{\phi_i}$ converges smoothly to a smooth extremal metric $\omega_{\phi}$ with extremal vector field $V.$ 
\end{thm}
\begin{proof}Note that the equation \eqref{equation-i0} differs Chen's continuity path for csck only by the term $\theta_i$, which is uniformly bounded. Hence the argument proceeds almost identical to Chen-Cheng \cite{CC3}[Section 3].
Denote 
\[
R_i=\underline{R}-\frac{1-t_i}{t_i}n+\theta_i, \beta_i=\frac{1-t_i}{t_i}\omega_i, (\beta_0)_i=\frac{1-t_i}{t_i}\omega, f_i=\frac{1-t_i}{t_i}h_i.
\]
Chen-Cheng \cite{CC3}[Lemma 3.11] prove that (using Tian's $\alpha$-invariant \cite{Tian97}),  for any $p>1$, there exists $\epsilon_p>0$ such that if $t_i\in (1-\epsilon_p, 1)$, one has
\[
\int_M e^{-pf_i}\omega^n\leq C,
\]
for a constant $C$ uniformly bounded (independent of $p$ due to the choice of $\epsilon_p$). Hence Chen-Cheng's Theorem \ref{keyestimate2} applies to get
the following estimate
\[
\|F_i+f_i\|_{W^{1, 2p}}+\|n+\Delta \phi_i\|_{L^p}\leq C_1,
\]
where $C_1=C_1(p, \omega, H_\infty)$ (see \eqref{entropy100}). By taking a subsequence, we can pass to the limit to get $T$-invariant functions $\phi_*\in W^{2, p}, F_{*}\in W^{1, p}$ for any $p<\infty$  such that
\begin{equation}
\begin{split}
&\phi_i\rightarrow \phi_*\; \text{in}\; C^{1, \alpha}\; \text{and}\; \sqrt{-1}\p\bar\p \phi_i\rightarrow \p\bar\p \phi_*\;\text{weakly in}\; L^p\\
&F_i+f_i\rightarrow F_{*}\;\text{in}\; C^{\alpha}\; \text{and}\; \nabla(F_i+f_i)\rightarrow \nabla F_*\;\text{weakly in}\; L^p\\
&\omega_{\phi_i}^k\rightarrow \omega_{\phi_*}^k\;\text{weakly in}\; L^p, 1\leq k\leq n.
\end{split}
\end{equation}
It follows that $\phi_*$ is a weak solution of extremal metric in the following sense,
\begin{equation}\label{vol100}\omega^n_{\phi_*}=e^{F_{*}}\omega^n,\end{equation}
and for any $u\in C^\infty(M)$, we have
\begin{equation}\label{v102}
-\int_M d^c F_*\wedge du\wedge \frac{\omega^{n-1}_{\phi_*}}{(n-1)!}=-\int_M u(\underline R+\theta_X(\phi_*))\frac{\omega_{\phi_*}^n}{n!}+uRic\wedge \frac{\omega^{n-1}_{\phi_*}}{(n-1)!}.
\end{equation}
Next we claim that $\int_M|f_i|\omega^n\rightarrow 0$. Given the claim, it follows that $e^{-f_i}\rightarrow 1$ in $L^p$ for any $p<\infty$ by a modified Lebesgue's dominated convergence theorem since $\sup_i\int_Me^{-p^{'}f_i}\omega^n<\infty$  (take $p<p^{'}$). 

By \eqref{j100},  $(1-t_i)\mathbb{J}(\tilde \phi_i)\rightarrow 0$ when $i\rightarrow \infty$. This implies that $(1-t_i)d_1(0, \tilde \phi_i)\rightarrow 0$. 
Note that we have the normalization condition $\sup h_i=0$. Denote $\tilde h_i=h_i-\mathbb{I}(h_i)\text{Vol}^{-1}(M)\in \cH_0$.  It follows that
\[
d_1(0, \tilde h_i)-d_1(0, \phi_i)\leq d_1(\tilde h_i,  \phi_i)=d_1(\sigma[0], \sigma[\tilde\phi_i])=d_1(0, \tilde \phi_i),
\]
in the last step we know that $G$ acts on $\cH_0$ isometrically. Hence we get $(1-t_i)d_1(0, \tilde h_i)\rightarrow 0$. By Theorem \ref{dar1}, 
$(1-t_i)\int_M |\tilde h_i|\omega^n\rightarrow 0.$ Since $\sup h_i=0$, we have for a uniformly bounded $C$, 
\[
0\leq \int_M (-h_i)\omega^n\leq C
\]
We have $\mathbb{I}(h_i)\leq 0$ and \[\mathbb{I}(h_i)-\int_M h_i\frac{\omega^n}{n!}=\mathbb{J}(h_i)\geq 0\]
It follows that $\mathbb{I}(h_i)$ is uniformly bounded, and hence $\int_M|(1-t_i)h_i|\omega^n\rightarrow 0$. This proves the claim $\int_M|f_i|\omega^n\rightarrow 0$. This in particular proves \eqref{vol100}. By \cite{chenhe12} we know that there exists a uniform constant $C_0>1$ such that
\begin{equation}\label{ch100}
C_0^{-1}\omega \leq \omega_{\phi_*}\leq C_0\omega, \phi_*\in W^{3, p}.
\end{equation}
By \eqref{equation-i0}, we have
\[\Delta_{\phi_i} (F_i+f_i)=(\underline R-\frac{1-t_i}{t_i}n+\theta_i)+\text{tr}_{\phi_i}(Ric(\omega)-\frac{1-t_i}{t_i}\omega)
\]
For any smooth function $u$, we write
\begin{equation*}
\int_M  (F_i+f_i) d^cdu\wedge \frac{\omega^{n-1}_{\phi_i}}{(n-1)!}=-\int_M u(\underline R-\frac{1-t_i}{t_i}n+\theta_X(\phi_i))\frac{\omega_{\phi_i}^n}{n!}+u(Ric-\frac{1-t_i}{t_i}\omega)\wedge \frac{\omega^{n-1}_{\phi_i}}{(n-1)!}.
\end{equation*}
Given the convergence of $F_i+f_i\rightarrow F_*$ in $C^\alpha$, $\omega_i^k$ converges weakly and $\phi_i$ converges to $\phi_*$ in $C^\alpha$, we can then pass to the limit to get
\begin{equation}\label{v105}
\int_M  F_* d^cdu\wedge \frac{\omega^{n-1}_{\phi_*}}{(n-1)!}=-\int_M u(\underline R+\theta_X(\phi_*))\frac{\omega_{\phi_i}^n}{n!}+u Ric\wedge \frac{\omega^{n-1}_{\phi_*}}{(n-1)!}.
\end{equation}
The standard elliptic theory together with \eqref{ch100} implies that $\phi_*$ is a smooth extremal metric with extremal vector field $V$. 
\end{proof}

By carefully checking the process above, the proof can be carried out in a $K$-invariant way with necessary modifications, for a maximal compact subgroup $K$ of $\text{Aut}_0(M, V)$.  Note that we only assume $\cK_X$ is bounded below on $\cH_K$, the $K$-invariant metrics. By assumption $G=\text{Aut}_0(M, V)$ is the complexification of $K$. 
Now we assume reductive-properness and we show that there exists an extremal metric. We indicate modifications required in the above arguments to solve \eqref{continuity1}. We choose $\omega$ to be $K$-invariant and 
apply  Theorem \ref{open1} and Corollary \ref{open2} in a $K$-invariant way (simply replacing $T$ by $K$) to get solutions which are $K$-invariant in an open interval $[0, t_0)$ for $t_0\leq 1$. Since $\cK_X$ is bounded below on $\cH_K$, then at any time $t\in [0, 1)$, for any $\phi\in \cH_K$, we have,
\[
t\cK_X(\phi)+(1-t)\mathbb{J}(\phi)\geq (1-t)\frac{1}{n+1}I(\phi)-C
\]
Lemma \ref{kbelow1} applies that the continuity path \eqref{continuity1} 
has a $K$-invariant solution for each $t\in [0, 1)$ by the same argument. Now we consider $t\rightarrow 1$. Let $\tilde \phi_i$ solve the equation \eqref{continuity1} at $t_i\rightarrow 1$. We  establish the following, 

\begin{lemma}\label{kenergy2}We have the following, 
\begin{equation}\label{j200}
\begin{split}
&\tilde K_{t_i}(\tilde \phi_i)=\inf_{\phi\in \cH_K} \tilde K_{t_i}(\phi)\rightarrow \inf_{\cH_K} \cK_X(\phi), t_i\rightarrow 1\\
&\cK_X(\tilde \phi_i)\rightarrow \inf_{\cH_K} \cK_X(\phi), t_i\rightarrow 1\\
&(1-t_i)\mathbb{J}(\tilde \phi_i)\rightarrow 0.
\end{split}
\end{equation}
\end{lemma} 
Using the fact that $\cK_X$ is bounded below over $\cH_K$, the same argument then applies, with $\inf_{\cH_X}$ replaced by $\inf_{\cH_K}$. Applying reductive-properness assumption to the sequence $\tilde \phi_i$, we get the distance bound $\sup_id_{1, G}(0, \tilde \phi_i)<\infty$. There exists $\sigma_i\in G$ with $\sigma_i^*\omega_{\tilde \phi_i}=\omega_{\phi_i}$ such that $\sup_i d_1(0, \phi_i)<\infty$. Note that $\phi_i$ and $\omega_{h_i}=\sigma_i^*\omega$ would not be $K$-invariant (only $T$-invariant) in general since $K$ is not a normal subgroup of $\text{Aur}_0(M, V)$. But $\cK_X$ is invariant under the action of $G$, by Lemma \ref{ginvariant01}. Hence Proposition \ref{e1000} and Lemma \ref{l100} both hold, such that $\phi_i$  solves the equation \eqref{equation-i0}, together with the uniform bounds $\sup_i d_1(0, \phi_i)<\infty, H_\infty<\infty$. By the same argument as in Theorem \ref{con100}, we get a smooth extremal metric $\phi_*$ as the limit of $\phi_i$ (passing to subsequence). The proof for the pair $(K_0, G_0)$ is identical. 

\subsection{Existence of an extremal metric implies properness modulo $G$}\label{p1000}
The main result of this section is to prove the following regularity result, which generalizes Chen-Cheng's result \cite{CC2}[Theorem 5.1] to extremal case. 
\begin{thm}\label{regularity1}Let $\phi_*\in \cE^1_X$ be a minimizer of $\cK_X$ over $\cE^1$. Then $\phi_*$ is smooth and it defines a smooth extremal metric. 
\end{thm}

\begin{rmk}It is a conjecture of Darvas-Rubinstein \cite{DR} that a minimizer of $\cK$-energy in $\cE^1$ is a smooth csck and it was proved by Chen-Cheng \cite{CC2}. Earlier the author, together with  Y. Zeng \cite{HZ}, proved partly Chen's conjecture that a $C^{1, 1}$ minimizer of $\cK$ is a smooth csck and it extends to extremal case.  
\end{rmk}

First we prove the following result, generalizing \cite{CC2}[Lemma 5.2]. 
\begin{lemma}Let $\phi_*\in \cE^1_X$ be a minimizer of $\cK_X$ over $\cE^1$. Then there exists a smooth extremal metric in $[\omega]$ with extremal vector field $V$.
\end{lemma}
\begin{proof}We want to solve the equation \eqref{continuity1} up to $t=1$. By Lemma \ref{kbelow1}, the equation has a smooth solution at $t\in [0, 1)$.  Choose a sequence $t_i\rightarrow 1$, and let $\phi_i$ be the solution at $t=t_i$, with normalization $\mathbb{I}(\phi_i)=0$. Note that $\phi_i$ minimizes $\tilde K_{t_i}$ over $\cE^1_X$, similar to \cite{CC2}[Corollary], replacing $\cK$-energy by $\cK_X$. Hence it follows that
\[
t_i\cK_X(\phi_*)+(1-t)\mathbb{J}(\phi_i)\leq t_i\cK_X(\phi_i)+(1-t)\mathbb{J}(\phi_i)\leq t_i\cK_X(\phi_*)+(1-t)\mathbb{J}(\phi_*)
\]
This implies $\mathbb{J}(\phi_i)\leq \mathbb{J}(\phi_*)$. Now $\mathbb{J}$ is proper in the sense $\mathbb{J}(\phi)\geq \delta d_1(0, \phi)-C$, for $\phi\in \cH_0$. It follows that
\[
\sup_i d_1(0, \phi_i)\leq \frac{1}{\delta}(C+\mathbb{J}(\phi_*))<\infty. 
\]
This in turn implies that, 
\[
\sup_i\int_M \log \frac{\omega_{\phi_i}^n}{\omega^n} \omega_{\phi_i}^n<\infty
\]
It then follows from Theorem \ref{keyestimate1}, that $\phi_i$ converges to a smooth function $\phi_\infty$ which defines a smooth extremal metric. 
\end{proof}

One can proceed by generalizing the strong uniqueness theorem of csck as in \cite{BDL2} to extremal case. Here we follow the proof of \cite{CC2}[Theorem 5.2] to prove Theorem \ref{regularity1}.

\begin{proof}Let $\phi_*$ be a minimizer of $\cK_X$ over $\cE^1_X$. By Lemma \ref{approx3} below, we can find a sequence $\phi_i\in \cH_X$ such that $d_1(\phi_i, \phi_*)\rightarrow 0$ and also the entropy converges, $H(\phi_i)\rightarrow H(\phi_*)$. Since $\mathbb{J}_{-Ric}, \mathbb{J}^X$ are all $d_1$ continuous, it follows that $\cK_X(\phi_i)\rightarrow \cK_X(\phi_*)$. 
We consider the following continuity path, for each $i$, 
\begin{equation}\label{continuity3}
t(R_\phi-\theta_X(\phi)-\underline{R})-(1-t)(\text{tr}_\phi \omega_{\phi_i}-n)=0. 
\end{equation}
The openness follows from Theorem \ref{open1} and Corollary \ref{open2} by taking $\alpha=\omega_{\phi_i}$. Denote the functionals $\mathbb{J}_i$, $\mathbb{I}_i, J_i, I_i$ to be the corresponding functionals with the base K\"ahler form as $\omega_{\phi_i}$. 
By definition we have
\[\mathbb{J}_i(\phi)=\mathbb{J}(\omega_{\phi_i}, \omega_\phi), I_i(\phi)=I(\omega_{\phi_i}, \omega_\phi), J_i(\phi)=J(\omega_{\phi_i}, \omega_\phi)\]
The corresponding functional for the path \eqref{continuity3} is
\[
\tilde K_t^i:=t\cK_X+(1-t)\mathbb{J}_i.
\]
Since for each $i$, $\mathbb{J}_i$ is proper, hence arguing exactly as in  Lemma \ref{kbelow1}, we can get a unique smooth solution $\phi_i^t$ to the equation \eqref{continuity3}, for $t\in [0, 1)$ with $\mathbb{I}_\omega(\phi_i^t)=0$. Note that $\phi^t_i$ minimizes $\tilde K_t^i$ over $\cH_X.$  It follows that for any $\phi\in \cE^1_X$
\begin{equation}\label{min1}
t\cK_X(\phi_i^t)+(1-t_i)\mathbb{J}_i(\phi_i^t)\leq t \cK_X(\phi)+(1-t)\mathbb{J}_i(\phi)
\end{equation}
Taking $\phi=\phi_i$ in \eqref{min1} (noting that $\phi_i$ minimizes $\mathbb{J}_i(\phi)$), we have
\begin{equation}\label{k01}
0=\mathbb{J}_i(\phi_i)\leq \mathbb{J}_i(\phi_i^t), \;\text{and hence}\; \cK_X(\phi_i^t)\leq \cK_X(\phi_i)
\end{equation}
Taking $\phi=\phi_*$ in \eqref{min1} (noting that $\phi_*$ minimizes $\cK_X$), we have
\[
\mathbb{J}_i(\phi_i^t)\leq \mathbb{J}_i(\phi_*).
\]
In other words, we have
\begin{equation}
0=\mathbb{J}_i(\phi_i)\leq \mathbb{J}_i(\phi_i^t)\leq \mathbb{J}_i(\phi_*).
\end{equation}
Hence we have, by the result of T. Darvas (see Theorem \ref{dar1}),
\[
\mathbb{J}_i(\phi_*)=(I_i-J_i)(\phi_*)\leq \frac{1}{n+1}I_i(\phi_*)\leq Cd_1(\phi_i, \phi_*)\rightarrow 0. 
\]
Note that we have,
\[
\mathbb{J}_i(\phi_i^t)=(I_i-J_i)(\phi_i^t)\geq  \frac{1}{n+1}I_i(\phi_i^t)
\]
Hence it follows that $I_i(\phi_i^t)\rightarrow 0$ when $i\rightarrow \infty$, uniform in $t$. By \cite{BBEGZ}[Theorem 1.8], we have
\[
I_\omega(\phi_i^t)\leq C_n(I_{\omega}(\phi_i)+I_i(\phi_i^t))\leq C
\]
This gives the distance bound  (see \cite{DR}[Proposition 5]), 
\[
d_1(0, \phi_i^t)\leq CI_{\omega}(\phi_i^t)+C.
\]
Together with the $\cK_X$-energy bound \eqref{k01}, this gives the uniform upper bound of the entropy $H(\phi_i^t)$, uniformly in $i$ and $t\in (0, 1)$. By Chen-Cheng's key estimate Theorem \ref{keyestimate1}, we conclude that $\phi^t_i$ converges to a smooth $u_i$ when $t\rightarrow 1$, such that $u_i\in \cH_X^0$ solves the equation
\[
R_{u_i}-\underline R-\theta_X(u_i)=0.
\]
In other words, $u_i$ defines an extremal metric for each $i$ and we have $I_i(u_i)\rightarrow 0$ when $i\rightarrow \infty$. Note that $d_1(0, u_i)$ is also uniformly bounded, this implies that by passing to a subsequence, if necessary $u_i$ converges to a smooth extremal potential $u\in \cH_X^0$. It follows that
\[
I(\omega_u, \omega_{\phi_*})\leq C \left(I_i(u_i)+I(\omega_{\phi_i}, \omega_{\phi_*})\right)\rightarrow 0. 
\] 
This implies that $u$ differs by $\phi_*$ by a constant (see \cite{CC2}[Lemma 5.7] for example). 
\end{proof}

We need an approximation in $d_1$ with convergent entropy in a $T$-invariant way, where $T$ is the compact subtorus  generated by $JV$.

\begin{lemma}\label{approx3}Given $u\in \cE^1_X$, there exists $u_k\in \cH_X$ such that $d_1(u, u_k)\rightarrow 0$ and $H_\omega(u_k)\rightarrow H_\omega(u)$, where the entropy is defined to be
$H_\omega(u)=\frac{1}{n!}\int_M \log(\frac{\omega_u^n}{\omega^n}) \omega^n$. 
\end{lemma}

\begin{proof}
Our argument is a modification of \cite{BDL}[Lemma 3.1]. We can suppose $H(u)<\infty$. Otherwise any $u_k\in \cH_X$ with $d_1(u_k, u)\rightarrow 0$ satisfies the requirement, since Fatou's lemma implies that $H$ is lower semi-continuous with respect to $d_1$ topology. Let $h=\frac{\omega_u^n}{\omega^n}\geq 0$. First we denote $h_k=h$ for $k^{-1}\leq h\leq k$, and $h_k=k$ for $h>k$ and $h_k=k^{-1}$ for $h<k^{-1}$.  Since $h$ is $T$-invariant, so is $h_k$. Clearly $h_k$ is a positive bounded function and $h_k\rightarrow h$ in $L^1$ and 
\[
\int_M h_k\log h_k\omega^n\rightarrow \int_M h\log h \omega^n. 
\]
We approximate $h_k$ by a smooth function $\tilde h_k$ which is bounded for each $k$ such that  $|\tilde h_k-h_k|_{L^1}<k^{-1}$  (the approximation can be taken in any $L^p$ norm for any $p<\infty$ since $h_k$ is uniformly bounded for each $k$) and
\begin{equation}\label{h01}
\left|\int_M h_k\log h_k\omega^n-\int_M \tilde h_k\log \tilde h_k \omega^n\right|<k^{-1}. 
\end{equation}
We need to choose $\tilde h_k$ such that it is $T$-invariant. Since $h_k$ and $\omega$ are both $T$-invariant, hence for any $\sigma\in T$, $\tilde h_k\circ \sigma$ is also an approximation of $h_k$ which satisfies \eqref{h01} in particular. Taking average with respect to $T$-action, that gives a $T$-invariant approximation, 
\[
h^T_k=\int_T \sigma^*(\tilde h_k) d\mu_\sigma
\] 
where $\mu$ is a $T$-invariant measure (a Haar measure) on $T$ with $\mu(T)=1$.  
Now we solve the Calabi-Yau equation
\[
\omega^n_{v_k}=\frac{h^T_k \omega^n}{\int_M h^T_k\omega^n}
\]
By the uniqueness (modulo constant) of the solution we know that $v_k\in \cH_X$. By \cite{BDL}[Theorem 2.8] we have $\tilde h\in \cH_X$ such that $d_1(v_k, h)\rightarrow 0$ (after passing to a subsequence). By \cite{D2}[Theorem 5(i)] we have $\omega_h^n=\omega_u^n$. The uniqueness \cite{GZ} then implies that $h=u$ modulo a constant. Hence we can suppose that $d_1(v_k, u)\rightarrow 0$ with the desired entropy approximation. 
\end{proof}

Lemma \ref{approx3} has an obvious generalization, for any connected compact subgroup $P$ of $\text{Aut}_0(M)$, we can have a smooth $P$-invariant approximation of $\cE^1_P$ with convergent entropy.
\begin{cor}Given $u\in \cE^1_P$, there exists $u_k\in \cH_P$ such that $d_1(u, u_k)\rightarrow 0$ and $H_\omega(u_k)\rightarrow H_\omega(u)$, where the entropy is defined to be
$H_\omega(u)=\frac{1}{n!}\int_M \log(\frac{\omega_u^n}{\omega^n}) \omega^n$. 
\end{cor}

As a direct application, we have the following,
\begin{thm}Suppose $(M, [\omega], J)$ admits an extremal K\"ahler metric with extremal vector field $V$. Then the modified $\cK_X$ energy is $d_{1, G}$ proper, as defined in Definition \ref{rp}. In particular we get two constants $C, D>0$ such that
\[
\cK_X(\phi)\geq Cd_{1, G}(0, \phi)-D.
\]
\end{thm}
\begin{proof}We argue that $\cK_X$ is $d_{1, G}$ proper in $\cE^1_X$.  We choose the base metric $\omega$ to be an extremal metric. Suppose otherwise, there exists $\phi_i\in \cE^1_{0, X}$ such that $\cK_X(\phi_i)\leq C$, but $d_{1}(0, \sigma[\phi_i])\rightarrow \infty$ for any $\sigma\in G$.  Connecting $0$ and $ \phi_i$ by a $d_1$ geodesic with unit speed (which is not unique in general \cite{D2}). Such a geodesic exists by Darvas's result \cite{D2}. Consider the point $u_i$ along the geodesic such that $d_1(0, u_i)=1$. The convexity of $\cK_X$ along the $d_1$ geodesic \cite{BDL} implies that 
\begin{equation}\label{p10}
\frac{\cK_X(u_i)}{d_1(0, u_i)}\leq \frac{\cK_X(\phi_i)}{d_1(0, \phi_i)}\rightarrow 0. 
\end{equation}
Given $\cK_X(u_i)\rightarrow 0$ and $d_1(0, u_i)=1$, the compactness result (which relies on results of \cite{D1, D2} and \cite{BBEGZ}) implies that $u_i$ converges to $u$ in $d_1$-topology (by subsequence), and $\cK_X(u)=0$ (lower semicontinuity). It follows that $u$ is a smooth extremal metric with extremal vector field by Theorem \ref{regularity1}. 

For $\tilde \phi_i\in \cE^1_{0, X}$ with $\tilde \phi_i=\sigma [\phi_i]$ for some $\sigma\in G$, such that $d_1(0, \tilde \phi_i)\rightarrow \infty$. Since $\cK_X$ is $G$-invariant, $\cK_X(\tilde \phi_i)\leq C$. The discussion above then applies to $\tilde \phi_i$. Connecting $0$ and $\tilde \phi_i$ by the geodesic with $\tilde u_i=\sigma[u_i]$, where $\tilde u_i$ is the point along the geodesic such that $d_1(0, \tilde u_i)=1$. By the discussion above, we get that $d_1(0, \tilde u)=1$ and $\tilde u=\sigma [u]$ such that $\cK_X(\tilde u)=0$.  Since this discussion holds for any $\sigma\in G$, this implies that $d_{1, G}(0, u)=1$. This contradicts the fact that $G$ acts transitively on extremal metrics with the extremal vector field $V$ \cite{BB, CLP}. By checking the proof more carefully,  when $d(x, \phi_i)\rightarrow \infty$, if we have $\frac{\cK_X(\phi_i)}{d_1(0, \phi_i)}\rightarrow 0,$
then \eqref{p10} still applies to get that $\cK_X(u)=0$, and it leads to contradiction. Hence, we can get constants $C, D>0$ such that, 
\[
\cK_X(\phi)\geq Cd_{1, G}(0, \phi)-D.
\]
\end{proof}

\begin{rmk}
Reductive properness and reduced properness are weaker than $G$-properness a priori. With the properness, one should be able to prove relative $K$-polystability if an extremal metric exists, in view of \cite{BDL2}. 
\end{rmk}

\end{document}